\documentclass[11pt,a4paper,reqno]{amsart}

\usepackage{times}
\usepackage{xspace}
\usepackage[inline]{enumitem}
\usepackage[numbers,sort&compress,longnamesfirst]{natbib}
\usepackage{geometry}
\usepackage[colorlinks=true, linkcolor=blue ,citecolor=blue]{hyperref}
\usepackage[capitalize]{cleveref}

\usepackage{amsfonts,comment}
\usepackage{amsmath,mathtools}
\usepackage{amssymb,amsmath,latexsym}
\numberwithin{equation}{section}
\usepackage{color, graphicx}
\usepackage{todonotes}

\def\<{\langle}
\def\>{\rangle}

\def\e{{\varepsilon}}

\newcommand{\be}{\begin{equation}}
\newcommand{\ee}{\end{equation}}
\newcommand{\bea}{\begin{eqnarray}}
\newcommand{\eea}{\end{eqnarray}}
\newcommand{\beas}{\begin{eqnarray*}}
\newcommand{\eeas}{\end{eqnarray*}}

\theoremstyle{plain}

\newtheorem{theorem}{Theorem}[section]

\newtheorem{prop}[theorem]{Proposition}

\newtheorem{corollary}[theorem]{Corollary}

\newtheorem{lemma}[theorem]{Lemma}
\newtheorem{examples}[theorem]{Examples}
\newtheorem{foo}[theorem]{Remarks}
\newtheorem{As}{Assumption}

\theoremstyle{definition}
\newtheorem{remark}[theorem]{Remark}
\newtheorem{example}[theorem]{Example}

\newenvironment{Example}{\begin{example}\rm}{\end{example}}

\newenvironment{Remark}{\begin{remark}\rm}{\end{remark}}

\newcommand{\Bi}{\boldsymbol{1}}

\newcommand{\cW}{\mathcal{W}}
\newcommand{\R}{\mathbb{R}}

\newcommand{\norm}[1]{\left\|#1 \right\|}     
\newcommand{\Inf}[1]{\inf\left\{#1 \right\}}
\newcommand{\abs}[1]{\left|#1\right|}

\newcommand{\upi}[1]{\left\lfloor#1\right\rfloor}

\newlist{myitems}{enumerate}{1}
\setlist[myitems]{label=\arabic*, font=\bfseries, resume}

\def\\ud sum{\displaystyle\sum}

 \def\T{\widetilde}  

 \def\ind#1{{\bf 1}_{\{#1\}}}

\numberwithin{equation}{section}

\def\P{{\mathbb P}}
\def\E{{\mathbb E}}

\def\lev{L\'evy\xspace}
\def\ue{\mathrm e}
\def\ud{\mathrm d}
\def\T{{\mathbb T}}

\allowdisplaybreaks

\begin{document}

\title[Convergence rates for BSDE\MakeLowercase{s} driven by L\'evy processes]{Convergence rates for Backward SDE\MakeLowercase{s}\\driven by L\'evy processes}

\author[C. Liu]{Chenguang Liu}
\author[A. Papapantoleon]{Antonis Papapantoleon}
\author[A. Saplaouras]{Alexandros Saplaouras}

\address{Delft Institute of Applied Mathematics, EEMCS, TU Delft, 2628 Delft, The Netherlands}
\email{C.Liu-13@tudelft.nl}

\address{Delft Institute of Applied Mathematics, EEMCS, TU Delft, 2628 Delft, The Netherlands \& Department of Mathematics, School of Applied Mathematical and Physical Sciences, National Technical University of Athens, 15780 Zografou, Greece \& Institute of Applied and Computational Mathematics, FORTH, 70013 Heraklion, Greece}
\email{a.papapantoleon@tudelft.nl}

\address{Department of Mathematics, School of Applied Mathematical and Physical Sciences, National Technical University of Athens, 15780 Zografou, Greece}
\email{alsapl@mail.ntua.gr}

\thanks{We thank Christel Geiss for a motivating discussion that initiated this project.
    AP gratefully acknowledges the financial support from the Hellenic Foundation for Research and Innovation Grant No. HFRI-FM17-2152.
    AS gratefully acknowledges the financial support from the Hellenic Foundation for Research and Innovation Grant No. 235 (2nd Call for H.F.R.I. Research Projects to support Post-Doctoral Researchers).}

\keywords{\lev processes, backward stochastic differential equations, compound Poisson approximation, $\mathbb L^2$-norm, Wasserstein distance, Blumenthal--Getoor index, convergence rate}  

\subjclass[2020]{60G51, 91G60, 60G44, 60G42}

\date{}

\begin{abstract}
We consider \lev processes that are approximated by compound Poisson processes and, correspondingly, BSDEs driven by \lev processes that are approximated by BSDEs driven by their compound Poisson approximations. 
We are interested in the rate of convergence of the approximate BSDEs to the ones driven by the \lev processes.
The rate of convergence of the \lev processes depends on the Blumenthal--Getoor index of the process.
We derive the rate of convergence for the BSDEs in the $\mathbb L^2$-norm and in the Wasserstein distance, and show that, in both cases, this equals the rate of convergence of the corresponding \lev process, and thus is optimal.
\end{abstract}

\maketitle

\section{Introduction}

Backward stochastic differential equations (BSDEs) have become an indispensable tool in stochastic analysis, because they allow us to describe phenomena that naturally arise in many applications. 
They also offer a direct link to other fields of mathematics, such as stochastic control, as an adjoint equation in the Pontryagin stochastic maximum principle, and non-linear partial differential equations, via a generalization of the Feynman--Kac theorem.
They also appear naturally in many applied fields, such as in mathematical finance, where they describe the hedging strategy for an option position or the price of an option subject to various valuation adjustments, in game theory, where their solutions characterize the value function of the game at equilibrium, or in energy and climate economics, where they can model, \textit{e.g.} the level of emissions and the price of allowance certificates.
A general overview of the theory and applications of BSDEs is beyond the scope of this article, hence we refer to the textbooks by \citet{Carmona_2016,Crepey_2013,Touzi_2013} and \citet{Zhang_2017} that cover both the theory and various of their applications.

\lev processes have been popular in mathematical finance for almost two decades now, because they allow to describe the reality in financial markets in an adequate way.
Indeed they can capture the discontinuities present in asset prices, model the fat-tails and skews present in asset log-returns under the `real-world' measure and, simultaneously, they exhibit an implied volatility smile under the `risk-neutral' measure.
Let us refer to the textbooks by \citet{Eberlein_Kallsen_2019, Cont_Tankov_2004} and \citet{Schoutens_2003} for the theory and applications of \lev processes in mathematical finance.

Backward SDEs and \lev processes, or, more generally, general semimartingales, have been combined already in several articles in order to extend the theory of BSDEs driven by classical Brownian motions to more general settings; see \textit{e.g.}  \citet{chitashvili1983martingale,buckdahn1993backward,elkaroui1997general,briand2002robustness,carbone2008backward,el1997backward} and \citet{papapantoleon2016existence} for results in settings where the generator is Lipschitz. 
The textbook by \citet{delong2013backward} offers an overview of the theory of BSDEs driven by jump processes, and their applications in insurance and finance.

The existence and uniqueness results for BSDEs driven by \lev processes and general semimartingales were naturally followed by approximation schemes for these type of equations. 
Let us mention here the articles by \citet{bouchard2008discrete}, \citet{aazizi2013discrete} in the pure jump case, \citet{lejay2014numerical} and \citet{geiss2016simulation} where the jump part of the driving martingale is a Poisson process, \citet{kharroubi2015decomposition} where the jump process depends on the Brownian motion itself, \citet{madan2015convergence} which follows the approach of \citet{briand2001donsker}, \citet{dumitrescu2016reflected} where the jump part of the driving martingale is a Poisson process and were reflected BSDEs are considered, \citet{khedher2016discretisation} for BSDEs driven by c\`adl\`ag martingales, and also \citet{Papapantoleon_Possamai_Saplaouras_2021} where discrete- and continuous-time BSDEs driven by general martingales are considered.

However, to the best of our knowledge, convergence rates for these schemes are only considered for BSDEs driven by Brownian motion, namely in the articles by \citet{briand2021donsker} and \citet{Geiss_Labart_Luoto_2020,geiss2020random}.

The aim of the present article is to derive convergence rates for BSDEs driven by \lev processes. 
More specifically, we consider \lev processes that are approximated by compound Poisson processes and, correspondingly, BSDEs driven by \lev processes that are approximated by BSDEs driven by their compound Poisson approximations. 
This approximation is very natural in this setting, since it gives rise to an exact simulation scheme for the compound Poisson processes.
As is well known, the rate of convergence of the \lev processes depends on the Blumenthal--Getoor index of the process, which encompasses information about the properties of the path of the \lev process.
We derive the rate of convergence for the BSDEs in the $\mathbb L^2$-norm and in the Wasserstein distance, and show that, in both cases, this is equal to the rate of convergence of the corresponding \lev processes, and thus optimal.
This is contrast to the results for Brownian motion, where the rate of convergence in the Wasserstein distance is optimal, but not in the $\mathbb L^2$-norm; see \citet{briand2021donsker}.

This article is organized as follows: in Section \ref{sec:approximation} we discuss the approximation of \lev processes by compound Poisson processes and derive the Blumenthal--Getoor index for popular classes of \lev models.
In Section \ref{sec:main}, we present the setting and the main results on the rate of convergence of BSDEs driven by \lev processes.
In Section \ref{sec:approx-BSDE}, we briefly discuss the case where the generator of the BSDEs is also approximated by another sequence.
In Section \ref{sec:optimality}, we prove that the rate of convergence in the Wasserstein distance is indeed optimal.
Finally, Section \ref{sec:main-proof} contains the proofs of the results, while Appendix \ref{appendix:neg-lev-rand-walk} contains an auxiliarry result on the approximation of \lev processes by random walks.

\section{Approximation of L\'evy processes}
\label{sec:approximation}

The aim of this section is to provide some auxiliary results on the approximation of \lev processes by compound Poisson processes, while \cref{appendix:neg-lev-rand-walk} contains a (negative) result on the approximation of \lev processes by random walks under the supremum norm.
These will be useful for determining the approximating process in the next section, when we will consider the approximation of BSDEs driven by \lev processes. 
 
Let $T>0$ be fixed, set $\mathbb T:=[0,T]$, and consider a complete stochastic basis $(\Omega, \mathcal{G}, \mathbb{G}, \P)$ in the sense of \citet[I.1.3]{jacod2003limit}, \textit{i.e.} the filtration $\mathbb{G}=(\mathcal{G}_t)_{t\in\T}$ satisfies the usual conditions.
Moreover, let $\E$ denote the expectation with respect to the measure $\P$.
Let us consider an infinite activity pure-jump \lev process, that will be approximated by a compound Poisson process. 
We want to set the notation and derive the rate of convergence for this approximation.

Let $X=(X_t)_{t\in\T}$ denote a pure-jump, square integrable L\'evy martingale with triplet $(0,0,\nu)$ and canonical decomposition
\begin{align}\label{levylim}
    X_t= \int_0^t \int_{\R^d} x\widetilde \mu(\ud s,\ud x),
\end{align}
where $\widetilde \mu(\ud s,\ud x)=  \mu(\ud s,\ud x)-\nu(\ud x)\ud s,$  and $\mu$ is the Poisson random measure associated with  $(\Delta X_t)_{t\in\T}$.
The square integrability of the process means that the following condition is satisfied: $ \int_{\norm{x}\ge1} \norm{x}^2 \nu(\ud x)<+\infty$.
Moreover, we will assume in the sequel that the filtration $\mathbb{G}$ is the usual augmentation of the natural filtration generated by the \lev process $X$.

Let us introduce an approximating sequence for this \lev martingale. 
Let $X^n=(X^n_t)_{t\in\T}$ be a pure-jump, square integrable \lev martingale with triplet $(0,0,\nu^n)$ and canonical decomposition
\begin{align}\label{levyapp}
    X^n_t= \int_0^t \int_{\R^d} x\widetilde \mu^n(\ud s,\ud x),
\end{align}
where $\widetilde \mu^n(\ud s,\ud x)=  \mu^n(\ud s,\ud x)-\nu^n(\ud x)\ud s,$  and $\mu^n$ is the Poisson random measure associated with  $(\Delta X^n_t)_{t\in\T}$, for every $n\in\mathbb N$.
The natural choice for the Poisson random measure $\mu^n$ is 
\begin{align}\label{can-app-jump-meas}
    \mu^n(\ud s,\ud x) := \ind{\norm{x}\ge \frac1n} \mu(\ud s,\ud x),
\end{align}
\textit{i.e.} we truncate the small jumps of the process in a ball of radius $\frac1n$ and send $n\to\infty$.
The definition of $\mu^n$ implies that, for every $n\in \mathbb{N}$, the associated \lev measure equals
\begin{align}\label{can-app-lev-meas}
    \nu^n(\ud x) = \ind{\norm{x}\ge \frac1n} \nu(\ud x).
\end{align}
This approach gives rise to a simulation scheme, since $X^n$ is a compound Poisson process that can be simulated exactly; see \textit{e.g.} \citet[\S 6.3]{Cont_Tankov_2004}.
Let $\mathbb{G}^n$ denote the filtration generated by the \lev martingale $(X^n_t)_{t\in\T}$.
Assuming that \eqref{can-app-jump-meas} holds, then $\mathbb G^n\subset \mathbb G^{n+1}\subset \dots \subset \mathbb G$ for every $n\in\mathbb{N}$.
Additionally, an immediate, nevertheless important, observation based on the special form of the random measure \eqref{can-app-jump-meas} is that every $\mathbb{G}^n-$martingale remains a $\mathbb{G}-$martingale.

The Blumenthal--Getoor index $\beta_*$ of the \lev process $X$, defined below in terms of the \lev measure $\nu$, 
\begin{align*}
      \beta_* := \Inf{ \beta>0,\  \int_{\norm{x}\leq 1} \norm{x}^\beta \nu(\ud x) <+\infty },
\end{align*}
plays a particular role in the computation of the convergence rate.

\begin{lemma}\label{appprocess}
Let $X$ and $X^n$ be as in \eqref{levylim} and \eqref{levyapp}--\eqref{can-app-lev-meas} and assume that the Blumenthal--Getoor index satisfies $\beta_*<2.$  
Then, for any $n\ge 1$ and $\beta\in(\beta_*,2)$, we have the following inequality
\begin{align*}
    \E\Big[\sup_{t\in\T}\norm{ X_t- X^n_t}^2\Big]^{\frac12}  \leq \frac{C_\beta \sqrt T}{n^{1-\frac\beta2}},
\end{align*}
where $ C_\beta = 2 \big(\int_{\R^d} \norm{x}^\beta \nu(\ud x)\big)^{\frac12}$.
\end{lemma}

\begin{proof}
Using the definition of $X^n$ and $X$, we have that
\begin{align*}
    X_t- X^n_t= \int_0^t \int_{\norm{x}\leq \frac{1}{n}} x\widetilde \mu(\ud s,\ud x),
\end{align*}
which is a martingale in the filtration $\mathbb G$.
Then, by Doob's inequality, we get for $\beta \in (\beta_*,2)$ that
\begin{align*}
 \E\Big[\sup_{t\in\T}\norm{ X_t- X^n_t}^2\Big]
    &\leq 4\E\bigg[\norm{\int_0^T \int_{\norm{x}\leq \frac{1}{n}} x\widetilde \mu(\ud s,\ud x)}^2\bigg]
     = 4\int_0^T \int_{\norm{x}\leq \frac{1}{n}} \norm{x}^2 \nu(\ud x)\ud s \\
    &= 4T \int_{\norm{x}\leq \frac{1}{n}} \norm{x}^{2-\beta}\norm{x}^{\beta}  \nu(\ud x)
     \leq 4T \int_{\norm{x}\leq \frac{1}{n}}\frac{1}{n^{2-\beta}} \norm{x}^\beta \nu(\ud x) \\
    &\leq \frac{4T}{n^{2-\beta}} \int_{\norm{x}\leq 1} \norm{x}^\beta \nu(\ud x).
\end{align*}
Setting $C_\beta$ as above and taking the square root on both sides completes the proof.
\end{proof}

\begin{remark}
The Blumenthal--Getoor index contains information about the variation of the paths of a \lev process. 
Moreover, the Blumenthal--Getoor index is strongly related to the Sobolev index, see \citet{Glau_2016}, that determines the smoothness of the distribution of the \lev process.
\end{remark}

Let us now compute the Blumenthal--Getoor index for certain popular classes of \lev processes.

\begin{example}[Generalized Hyperbolic process]
The L\'evy measure of the generalized hyperbolic (GH) distribution has the following form:
\begin{align*}
\nu^{\text{GH}}(\ud x)
    = \frac{\ue^{\gamma x}}{\abs{x}} \bigg( \int^\infty_0 \frac{ \exp\big(-\sqrt{2y+\alpha^2}\abs{x}\big)}{\pi^2y\{J^2_{\abs{\lambda}}(\delta\sqrt{2y})+Y^2_{\abs{\lambda}}(\delta\sqrt{2y})\}}\ud y
        + \lambda \ue^{-\alpha \abs{x}}\boldsymbol{1}_{\{\lambda\ge 0\}}\bigg) \ud x,
\end{align*}
where $\alpha,\delta>0$ and $\gamma\in(-\alpha,\alpha)$, while $J_{\abs{\lambda}}, Y_{\abs{\lambda}}$ denote the modified Bessel functions of the first, resp. second, kind with index $|\lambda|$; see \textit{e.g.} \citet[Chapter 2]{Eberlein_Kallsen_2019}.
Using \citet[Proposition 2.18]{Raible2000LvyPI}, we have that the L\'evy measure of the GH process behaves like $ \boldsymbol{1}_{\{\abs{x}\leq 1\}}\nu^{\text{GH}}(\ud x)\sim \abs{x}^{-2}\boldsymbol{1}_{\{\abs{x}\leq 1\}}\ud x$. 
Therefore, we get that the Blumenthal--Getoor index equals $\beta_*=1$, since for any $\beta>1$,
\begin{align*}
   \int_{\abs{x}\leq 1}\abs{x}^{-2+\beta}\ud x<+\infty.
\end{align*}
\end{example}
 
\begin{example}[CGMY process]
The L\'evy measure of the CGMY L\'evy process equals 
\begin{align*}
\nu^{\text{CGMY}}(\ud x) = \frac{C}{x^{1+Y}} \Big( \ue^{-Mx}\boldsymbol{1}_{\{x>0\}} + \ue^{Gx}\boldsymbol{1}_{\{x<0\}} \Big) \ud x,
\end{align*}
where $C, G, M >0,$ and $Y <2$; see again \cite[Chapter 2]{Eberlein_Kallsen_2019}.
Obviously, we have that $\boldsymbol{1}_{\{\abs{x}\leq 1\}}\nu^{\text{CGMY}}(\ud x) \sim \abs{x}^{-1-Y}\boldsymbol{1}_{\{\abs{x}\leq 1\}}\ud x$.
Therefore, the Blumenthal--Getoor index equals $\beta_*=\max\{0, Y\}$, since for any $\beta>Y,$
\begin{align*}
   \int_{\abs{x}\leq 1}\abs{x}^{-1-Y+\beta}<+\infty.
\end{align*}
\end{example}

\begin{example}[Meixner process]
The L\'evy measure of the Meixner process equals 
\begin{align*}
    \nu^{\text{Meixner}}(\ud x)= \frac{\delta \exp \big(\frac{\beta}{\alpha}x\big)}{x \sinh\big(\frac{\pi}{\alpha}x\big)} \ud x,
\end{align*}
with $\alpha,\delta>0$ and $\beta\in(-\pi,\pi)$; see \citet{Schoutens_2003}.
We have that $\sinh\big(\frac{\pi}{\alpha}x\big)\sim x$ when $\abs{x}\leq 1,$ hence $ \nu^{\text{Meixner}}(\ud x)\sim \abs{x}^{-2}\ud x.$ 
This implies that the Blumenthal--Getoor index equals $\beta_*=1.$
\end{example}

\begin{example}[Pure-jump Merton model]
The canonical decomposition of the pure-jump Merton model is
\begin{align*}
    X_t = \sum_{k=1}^{N_t}J_k, 
\end{align*}
where $N$ is a Poisson process with parameter $\lambda>0$, while $J_k$ follows a normal distribution $\mathcal{N}(\mu, \sigma^2 ),$ for $k\ge1$; \textit{cf.} \cite[Chapter 2]{Eberlein_Kallsen_2019}. 
The L\'evy measure of the Merton model equals 
\begin{align*}
\nu^{\text{Merton}}(\ud x) = \frac{\lambda}{\sigma\sqrt{2\pi}}\exp\Big(-\frac{x^2}{2\sigma^2}\Big) \ud x.
\end{align*}
This implies $ \boldsymbol{1}_{\{\abs{x}\leq 1\}}\nu^{\text{Merton}}(\ud x) \sim \boldsymbol{1}_{\{\abs{x}\leq 1\}}\ud x$. 
Therefore, the Blumenthal--Getoor index equals $\beta_*=0.$
\end{example}

\section{Setting and main results}
\label{sec:main}

The aim of this section is to describe the setting we will employ, as well as the main results on convergence rates for BSDEs driven by \lev processes.
The starting point is the convergence result for BSDEs by \citet[Theorem 3.1]{Papapantoleon_Possamai_Saplaouras_2021}.
Based on the results of the previous section and Appendix \ref{appendix:neg-lev-rand-walk}, we choose a sequence of compound Poisson processes $(X^n)_{n\in\mathbb N}$ to approximate the \lev martingale $X$ driving the BSDE. 
The approximating process is quasi-left-continuous, therefore the integrator in the generator of the approximating BSDE in \cite[Theorem 3.1]{Papapantoleon_Possamai_Saplaouras_2021} can be chosen continuous, and in particular it can be the Lebesgue measure.

Let us thus consider the following BSDE driven by the \lev martingale $X$:
\begin{align}\label{mainlimeq}
Y_t = \xi + \int_t^T f\big(s,Y_s,U_s(\cdot)\big)\ud s - \int_t^T \int_{\R^d} U_s(x) \widetilde \mu(\ud s,\ud x),
\end{align}
where the terminal value $\xi$ is $\mathcal{G}_T$-measurable and $\E[\abs{\xi}^2]<+\infty.$
The approximating sequence for this BSDE is driven by the process $X^n$ and satisfies
\begin{align}\label{mainappeq}
Y^n_t = \xi^n + \int_t^T f\big(s,Y^n_s,U^n_s(\cdot)\big)\ud s - \int_t^T \int_{\R^d} U^n_s(x) \widetilde \mu^n(\ud s,\ud x),\ \  n\in\mathbb{N},
\end{align}
where the terminal value $\xi^n$ is $\mathcal{G}^n_T$-measurable and $\E[\abs{\xi^n}^2]<+\infty$, for any $n\in \mathbb N$.

 We will first assume that the terminal condition is general, \textit{i.e.} as described above, and will derive a general result on the rate of convergence.
Afterwards, we will consider the case where the terminal value is a Markovian function of the \lev process, in order to derive an explicit convergence rate.

\subsection{Setting}
\label{set-2}

The following assumptions will be in force at certain stages of this work.

\begin{myitems}[label=\textup{{(S\arabic*)}},itemindent=0.cm,leftmargin=*]
\item \label{tF2} The generator of equation \eqref{mainlimeq}, $f: [0,T]\times \R\times \mathbb L^2(\nu)\to\R$, satisfies a globally Lipschitz condition, uniformly for $t\in[0,T]$, \textit{i.e.} for any $y,y', z,z'$
    \begin{align*}
        \abs{f(t,y,z(\cdot))-f(t,y',z'(\cdot))}\leq L_f\Big[\abs{y-y'}+\Big(\int_{\R^d} \abs{z(x)-z'(x)}^2\nu (\ud x)\Big)^\frac{1}{2}\Big].
    \end{align*}
\item \label{G1} The terminal values satisfy $\xi=g(X_T)$ and $\xi^n=g(X^n_T)$, where the function $g:\R^d\to\R$ is Lipschitz, \textit{i.e.}
    \[
        \abs{g(x)-g(x')}\leq L_g\norm{x-x'}. 
    \]
\end{myitems}

\begin{remark}
The structure of the approximating process $X^n$ in \eqref{levyapp}--\eqref{can-app-jump-meas}, which implies that $\nu^n (\ud x) = \ind{\norm{x}\ge \frac1n} \nu(\ud x)$, immediately yields that $\mathbb{L}^2(\nu) \subset \mathbb{L}^2(\nu^n)$. 
Moreover, for every $h\in \mathbb{L}^2(\nu^n)$ it is also true that $h\ind{\norm{x}\ge \frac1n}\in\mathbb{L}^2(\nu)$.
This observation and Assumption \ref{tF2} imply an analogous Lipschitz condition for $\nu^n$,
\textit{i.e.} for $f: [0,T]\times \R\times \mathbb L^2(\nu^n)\to\R$ and any $y,y', z,z'$ holds that
\begin{align*}
    \abs{f(t,y,z(\cdot))-f(t,y',z'(\cdot))}\leq L_f\Big[\abs{y-y'}+\Big(\int_{\R^d} \abs{z(x)-z'(x)}^2\nu^n
    (\ud x)\Big)^\frac{1}{2}\Big].
\end{align*}
\end{remark}

\begin{Remark}
We have written the Lipschitz property in its usual form, while in \citet[Condition \textbf{(F3)}]{papapantoleon2016existence} the (stochastic) Lipschitz property is written in quadratic terms. 
We can easily verify that \ref{tF2} implies the quadratic form.
Indeed, using the notation of \cite[Lemma 2.13, Conditions \textbf{(F3)}]{papapantoleon2016existence} we have
\begin{align*}
    \abs{f(t,y,z(\cdot)) - f(t,y',z'(\cdot))}^2 
    & \leq 2L_f^2 \Big[\abs{y-y'}^2+\Big(\int_{\R^d} \abs{z(x)-z'(x)}^2\nu^n (\ud x)\Big)\Big]\\
    & = 2L_f^2 \Big[\abs{y-y'}^2+|\!|\!| z(x)-z'(x)|\!|\!|^2\Big].
\end{align*}
Moreover, the Lipschitz constants are deterministic and the Lebesgue integrator is atomless.
In other words, for $A,\Phi$ as described in \cite[Conditions \textbf{(F4)}]{papapantoleon2016existence}, the process $A$ is continuous, bounded and deterministic with $\Phi=0$.
Consequently,  given that 
\begin{align*}
    \int_0^T |f(t,0,0)|^2 \ud t <\infty,
\end{align*}
we can legitimately use \cite[Theorem 3.5]{papapantoleon2016existence} in order to conclude the existence and uniqueness of solutions for the BSDEs in \eqref{mainlimeq} and \eqref{mainappeq} in the present setting.
\end{Remark}

\subsection{Main results}

This subsection contains the main results of this work, which concerns convergence rates for BSDEs driven by \lev processes in the $\mathbb L^2$-norm and in the Wasserstein distance.

\begin{theorem} 
\label{maintheorem1}
Let $(Y,U)$ and $(Y^n,U^n)$ be the solutions of equations \eqref{mainlimeq} and \eqref{mainappeq}. 
Let $f$ satisfy Assumption \ref{tF2}, and assume that the approximating process $X^n$ has the form \eqref{levyapp}--\eqref{can-app-lev-meas}.
Then, there exists a constant $C=C_{L_f,T}$ such that
\begin{align}\label{bound Y1}
     \E\Big[\sup_{t\in\T}\abs{ Y^n_t- Y_t}^2\Big]^\frac{1}{2}
     \leq   C \E\big[ \abs{\xi^n-\xi}^2 \big]^\frac{1}{2},
\end{align}
and
\begin{align}
    \bigg( \int_0^T \int_{\R^d} \E\Big[\abs{\overline{U}^n_t(x)- U_t(x)}^2\Big]  \nu(\ud x)\ud t \bigg)^{\frac12} 
    \leq  C \E\big[ \abs{ \xi^n- \xi}^2 \big]^\frac{1}{2}. 
\end{align}
\end{theorem}

The proof of this result is deferred to \cref{sec:main-proof}.

\begin{Remark}
In the statement of the previous theorem, the approximating process $X^n$ depends on the measurable set $\{\norm{x}\ge\frac1n\}$, $n\in\mathbb{N}$. 
However, this set does not (explicitly) appear in the rates.
Essentially, this information is hidden in the convergence of the terminal values, which -- in turn -- encode the convergence of the $\sigma-$algebrae $\mathcal{G}^n_T \xrightarrow[n\to\infty]{\textup{w}} \mathcal{G}_T$. 
\end{Remark}

Assuming now that the terminal condition is a Markovian function of the \lev process, we can combine the results of \cref{maintheorem1} with \cref{appprocess} and compute an explicit convergence rate.

\begin{corollary}\label{cor-exp-rate}
Let $(Y,U)$ and $(Y^n,U^n)$ be the solutions of equations \eqref{mainlimeq} and \eqref{mainappeq}, where $f$ satisfies Assumption \ref{tF2} and $g$ satisfies Assumption \ref{G1}. 
Let $X$ and $X^n$ be as in \eqref{levylim} and \eqref{levyapp}--\eqref{can-app-lev-meas} and assume that the Blumenthal--Getoor index satisfies $\beta_*<2.$  
Then, for any $\beta\in(\beta_*,2)$, there exists a constant $C'=C_{L_f,L_g,T}' = C_{L_f,T}L_g$ (where $C_{L_f,T}$ is the constant from \cref{maintheorem1}), such that
\begin{align}
     \E\Big[\sup_{t\in\T}\abs{ Y^n_t- Y_t}^2\Big]^\frac{1}{2} 
     \leq C' \frac{C_{\beta}\sqrt{T}}{n^{1-\frac\beta2}},
\end{align}
and     
\begin{align}
    \bigg( \int_0^T \int_{\R^d} \E\Big[\abs{\overline{U}^n_t(x)- U_t(x)}^2\Big]  \nu(\ud x)\ud t \bigg)^{\frac12} 
    \leq C' \frac{C_\beta \sqrt{T}}{n^{1-\frac\beta2}}.
\end{align}
\end{corollary}

\begin{proof}
Using Assumption \ref{G1}, we have immediately that
\[
    \E\big[ \abs{\xi^n-\xi}^2 \big]^\frac{1}{2} 
        = \E\big[ \abs{g(X^n_T)-g(X_T)}^2 \big]^{\frac{1}{2}}
        \leq L_g \E\big[ \abs{X^n_T-X_T}^2 \big]^{\frac{1}{2}}.
\]
A direct application of \cref{appprocess} concludes the proof.
\end{proof}

\begin{remark}
The rate of convergence in the $\mathbb L^2$-norm is obviously optimal, as it coincides with the rate for the approximation of the \lev process itself.
\end{remark}

Moreover, we want to deduce the rate of convergence in the Wasserstein distance. 
The following result shows that, in the case of BSDEs driven by pure-jump \lev processes, the rate of convergence is the same in the $\mathbb L^2$-norm and in the Wasserstein distance, while in Section \ref{sec:optimality} we argue that this result is optimal. 
On the contrary, for BSDEs driven by Brownian motion the rate in the $\mathbb L^2$-norm is $n^{-\frac14}$, while in the Wasserstein distance it equals $n^{-\frac12}$; see \citet{geiss2020random, Geiss_Labart_Luoto_2020} and \citet{briand2021donsker}.

Let $v, v'$ be two probability measures on $(D(\mathbb T), \mathcal{B}(D(\mathbb T)))$, where $D(\mathbb T)$ is the space of c\`adl\`ag paths on $\mathbb T=[0,T]$. 

Define the Wasserstein distance between these measures by
\[
   \cW_\rho(v,v') =  \inf_{\Gamma\in\mathcal{H}(v,v') }\Big(\int_{D(\mathbb T)\times D(\mathbb T)}\rho^2(y,y')\Gamma(\mathrm{d}y,\mathrm{d}y')\Big)^{\frac{1}{2}},
\]
where $\rho(y,y'):=\sup_{t\in\T}{\abs{y_t-y_t'}}$ and $\mathcal{H}(v,v')$ is the set of couplings between $v$ and $v'$, \textit{i.e.}
$$\mathcal{H}(v,v') = \big\{ 
    \Gamma \in \text{Pr} (D(\mathbb T)\times D(\mathbb T)) \ | \
    \Gamma(A \times D(\mathbb T)) = v(A),  
    \Gamma(D(\mathbb T) \times B) = v'(B), \ \forall 
    A, B \in \mathcal{B}(D(\mathbb T))
\big\}.$$
Similarly, let $U,U'$ be two probability measures on $\widetilde D(\mathbb T)$, 
where $\widetilde D(\mathbb T)= \mathbb L^2(\nu(\mathrm{d}x)\times \mathrm{d}t).$   Define the Wasserstein distance between these measures by
\[
   \cW_{\widetilde \rho}(u,u') =  \inf_{\widetilde\Gamma\in \mathcal{\widetilde H}(u,u') }\Big(\int_{\widetilde D(\mathbb T)\times \widetilde D(\mathbb T)}\widetilde\rho^2(y,y')\widetilde \Gamma(\mathrm{d}y,\mathrm{d}y')\Big)^{\frac{1}{2}},
\] 
where $\widetilde \rho^2(y,y'):= \int_0^T \int_{\R^d}\abs{y_t'(x)- y_t(x)}^2  \nu(\ud x)\ud t$ and $\mathcal{\widetilde H}(u,u')$ is the set of couplings between $u$ and $u'$, \textit{i.e.}
$$\mathcal{\widetilde H}(u,u') = \big\{ 
    \Gamma \in \text{Pr} ( \widetilde D(\mathbb T)\times \widetilde D(\mathbb T)) \ | \
   \widetilde \Gamma(A \times \widetilde D(\mathbb T)) = u(A),  
    \widetilde \Gamma(  \widetilde D(\mathbb T) \times B) = u'(B), \ \forall 
    A, B \in \mathcal{B}( \widetilde D(\mathbb T))
\big\}.$$

\begin{corollary}
Let $(Y,U)$ and $(Y^n,U^n)$ be the solutions of equations \eqref{mainlimeq} and \eqref{mainappeq}, where $f$ satisfies Assumption \ref{tF2} and $g$ satisfies Assumption \ref{G1}.
Let $X$ and $X^n$ be as in \eqref{levylim} and \eqref{levyapp}--\eqref{can-app-lev-meas} and assume that the Blumenthal--Getoor index satisfies $\beta_*<2.$  
Then, for any $\beta\in(\beta_*,2)$, there exists a constant $C'$ (the same as in \cref{cor-exp-rate}), such that
\begin{align}\label{inequality wdy}
     \cW_\rho\big(Y^n,Y\big) 
     \leq C' \frac{C_\beta\sqrt{ T}}{n^{1-\frac\beta2}},\\
     \cW_{\widetilde\rho}\big(U^n,U\big) 
     \leq C' \frac{C_\beta\sqrt{ T}}{n^{1-\frac\beta2}}.
\end{align} 
\end{corollary}

\begin{proof}
Using the properties of the Wasserstein distance and \cref{cor-exp-rate}, we get immediately that
\begin{equation*}
\cW_\rho\big( Y^n,Y \big)
    \leq \E\Big[\sup_{t\in\T}\abs{ Y^n_t- Y_t}^2\Big]^\frac{1}{2}
    \leq C_{L_f,L_g,T}' \frac{C_\beta \sqrt{T}}{n^{1-\frac\beta2}},
\end{equation*}
and
\begin{equation*}
\cW_{ \widetilde \rho}\big( U^n,U \big)
    \leq  \Big( \int_0^T \int_{\R^d} \E\Big[\abs{\overline{U}^n_t(x)- U_t(x)}^2\Big]  \nu(\ud x)\ud t\Big)^{\frac{1}{2}}
    \leq C_{L_f,L_g,T}' \frac{C_\beta \sqrt{T}}{n^{1-\frac\beta2}}. \qedhere
\end{equation*}
\end{proof}

\section{Approximation of the BSDE generator}
\label{sec:approx-BSDE}

The aim of this section is to briefly discuss what happens in case the generator of the approximating BSDE $(Y^n)_{n\in\mathbb N}$ is not the same as the generator of the BSDE $Y$, but instead is a function $f^n$ that approximates the generator $f$.
The main question then is whether this approximation will affect the rate of convergence or not. 

Let us consider the BSDE $Y$ given by \eqref{mainlimeq}, while the approximating BSDE $Y^n$ takes now the form
\begin{align}\label{mainappeq1}
Y^n_t = \xi^n + \int_t^T f^n\big(s,Y^n_s,U^n_s(\cdot)\big)\ud s - \int_t^T \int_{\R^d} U^n_s(x) \widetilde \mu^n(\ud s,\ud x),\ \  n\in\mathbb{N}.
\end{align}
In addition to \ref{tF2} and \ref{G1}, the following assumptions will be used in this part of the work.

\begin{myitems}[label=\textup{{(S\arabic*)}},itemindent=0.cm,leftmargin=*]
\item \label{tF3} The generator of equation \eqref{mainappeq1}, $f^n: [0,T]\times \R\times \mathbb L^2(\nu^n)\to\R$, satisfies a globally Lipschitz condition, uniformly for $t\in[0,T]$, \textit{i.e.} for any $y,y', z,z'$
\begin{align*}
    \abs{ f^n(t,y,z(\cdot)) - f^n(t,y',z'(\cdot)) } &\leq  L_f \Big[ \abs{y-y'} + \Big(\int_{\R^d} \abs{z(x)-z'(x)}^2\nu^n (\ud x)\Big)^\frac{1}{2} \Big].
\end{align*}
\item \label{tF4} The generators $f^n$ and $f$ satisfy the following condition: for any $(t,y)\in [0,T]\times\R$ and any uniformly Lipschitz continuous function $u$ (\textit{i.e.}, $u$ for which $\abs{u(x)}/\abs{x}$ is bounded over all $x\in \R^d \backslash  \{0\}$) , holds
\begin{align}\label{usupport1}
    \big| f^n(t,y,u(\cdot))-f(t,y,u(\cdot)) \big| \leq c_n\Big(1+\abs{y}^2+\int_{\R^d} \abs{u(x)}^2\nu (\ud x)\Big)^\frac{1}{2},
\end{align}
for some $c_n>0$, that depends on $n\in\mathbb N$. 
Moreover, the sequence $(c_n)_{n\in\mathbb N}$ is assumed bounded.

\end{myitems}

\begin{theorem}
\label{maintheorem11}
Let $(Y,U)$ and $(Y^n,U^n)$ be the solutions of equations \eqref{mainlimeq} and \eqref{mainappeq1}, where $f$ and $f^n$ satisfy Assumptions \ref{tF2}, \ref{tF3} and \ref{tF4}, while $g$ satisfies Assumption \ref{G1}.
Let $X$ and $X^n$ be as in \eqref{levylim} and \eqref{levyapp}--\eqref{can-app-lev-meas} and assume that the Blumenthal--Getoor index satisfies $\beta_*<2.$  
Then, for any $\beta\in(\beta_*,2)$, there exists a constant $C''$, such that
\begin{align}\label{bound Y11}
    \E\Big[ \sup_{t\in\T}\abs{ Y^n_t- Y_t}^2 \Big]^\frac{1}{2} \leq C'' \bigg\{ \frac{C_{\beta}\sqrt{T}}{n^{1-\frac\beta2}} + c_n \bigg\},
\end{align}
and
\begin{align*}
    \bigg( \int_0^T \int_{\R^d} \E\Big[\abs{\overline{U}^n_t(x)- U_t(x)}^2\Big]  \nu(\ud x)\ud t \bigg)^{\frac12} \leq  C'' \bigg\{ \frac{C_{\beta}\sqrt{T}}{n^{1-\frac\beta2}} + c_n \bigg\}. 
\end{align*}
\end{theorem}

The proof of this result is deferred to \cref{sec:main-proof}.
Next, we present two examples of situations where an approximation of the generator arises naturally, and discuss the implications for the convergence rate. 

\begin{example}
\label{ex:1}
Assume that the generator $f$ of the BSDE $Y$ takes the following form
\begin{align*}
    f(t,y,u(\cdot)) = \int_{\R^d} \Phi(t,y,u(x)) \delta(x) \nu(\ud x),
\end{align*}
where $\delta(x) = 1\land\norm{x}^{\bar{\beta}}$, for $\bar{\beta}$ such that 
\[
\int_{\norm x \leq 1} \norm x^{\bar{\beta}} \nu(\ud x) <+\infty,
\]
and $\abs{\Phi(t,y,z)}\leq \widetilde C(1+\abs{y}+\abs{z}),$ for $\widetilde C>0$. 
Obviously, $\bar{\beta} > \beta_*$.
Considering the form of the \lev measure in \eqref{can-app-lev-meas}, we set the generator $f^n$ of the approximating BSDE $Y^n$ equal to
\begin{align*}
    f^n (t,y,u(\cdot)) = \int_{\R^d}\Phi(t,y,u(x))\delta(x) \nu^n(\ud x).
\end{align*}
Then, for $C_\beta$ as defined in \cref{appprocess}, we can verify that
\begin{align*}
\big| f^n(t,y,u(\cdot)) - f(t,y,u(\cdot)) \big|    
    &= \abs{ \int_{\norm{x}\leq \frac{1}{n}} \Phi(t,y,u(x))\delta(x) \nu(\ud x)}
    \leq  \int_{\norm{x}\leq \frac{1}{n}} \abs{ \Phi(t,y,u(x))\delta(x) } \nu(\ud x) \\
    &\leq \widetilde C \int_{\norm{x}\leq \frac{1}{n}} (1+\abs{y}+\abs{u(x)})\delta(x) \nu(\ud x)\\
    &\leq \widetilde C \, c_n \{ 1+\abs{y}+\norm{u}_{\mathbb L^2(\nu)} \},
\end{align*}
where 
\begin{align*}
    c_n := \max\Big\{ \int_{\norm{x}\leq \frac{1}{n}}\norm{x}^{\bar \beta} \nu(\ud x), \Big(\int_{\norm{x}\leq \frac{1}{n}}\norm{x}^{2\bar\beta} \nu(\ud x)\Big)^{\frac{1}{2}} \Big\}\leq C_\beta\max\Big\{\frac{1}{n^{\bar \beta-\beta}},\frac{1}{n^{\bar \beta-\beta/2}}\Big\}= \frac{C_\beta}{n^{\bar \beta-\beta}},
\end{align*}
for any $\beta\in(\beta_*,\bar \beta)$. 
In other words, the rate of convergence will become worse in case $\bar\beta < 1+\frac{\beta}{2}$.
\end{example}

\begin{remark}
The example above is inspired by the portfolio liquidation problem studied in \citet{kruse2016minimal}, where the function $\Phi$ takes the form
\begin{align}
    \Phi(t,y,u(x)) = (y+u(x))\left(1-\frac{\lambda_t(x)}{\big( (y+u(x))^{q-1} + \lambda_t(x)^{q-1} \big)^{p-1}}\right)\ind{y+u(x)\ge0},
\end{align}
where $\lambda$ is deterministic and time-dependent (in our framework), and $p,q$ are H\"older conjugates; see \cite[eq. (24)]{kruse2016minimal}.
This function obviously satisfies the conditions of the previous example. 
Note that we have omitted the term $Y^q$ from the driver, since we consider Lipschitz BSDEs.
\end{remark}

\begin{example}[Time discretization]
Assume that the generator $f$ of the BSDE $Y$ is an $\alpha$-H\"older continuous function in time for $0<\alpha\leq 1$,  \textit{i.e.} for any $t,s\in \T$ holds
\begin{align*}
    \abs{ f(t,y,u(\cdot)) - f(s,y,u(\cdot)) } \leq C\abs{t-s}^\alpha \{1+\abs{y}+\norm{u}_{L^2(\nu)}\}.
\end{align*}
Then, we set the generator $f^n$ of the approximating BSDE $Y^n$ equal to
\begin{align*}
    f^n(t,y,u(\cdot)) = f(t_i^n,y,u(\cdot)), \quad t_i^n\leq t\leq t_{i+1}^n,
\end{align*}
where $t_i^n = \frac{iT}{n}$ and $i=0,\dots,n-1.$ 
In this case, we can verify that $c_n= (\frac{T}{n})^\alpha.$
In other words, the rate of convergence will become worse in case $\alpha < 1-\frac{\beta}{2}$.
\end{example}

\section{Optimality}
\label{sec:optimality}

The aim of this section is to discuss the optimality of the convergence rate in the Wasserstein distance, \textit{i.e.} inequality \eqref{inequality wdy}.
In order to make the argument clear, we just consider the following special case:  the dimension $d = 1$, the generator $f\equiv0$ and the terminal conditions equal $\xi=X_T$ and $\xi^n=X^n_T$. 
In this case, the solutions in \eqref{mainlimeq} and \eqref{mainappeq} are exactly $(Y_t)_{t\in\T}=(X_t)_{t\in\T}$ and $(Y^n_t)_{t\in\T}=(X^n_t)_{t\in\T}$.
Then we have the following results.

\begin{lemma}\label{lem indep}
Let $X$ be a \lev process as in \eqref{levylim} and $\tau$ be its first jump time, \emph{i.e.} $\tau:=\inf\{t>0, \abs{\Delta X_t}>0\}.$ 
Then $\tau$ is independent from the first jump size $\abs{\Delta X_\tau}$. 
\end{lemma}

\begin{proof}
\textit{Compound Poisson processes:}
In this case, we have that $\nu([-1,1]^d) <+\infty,$ which implies that $(X_t)_{t\in\T}$ is a compound Poisson process. 
Hence, by the construction,  the first jump time is independent from the first jump size.

\textit{Infinite activity processes:}
In this case, we have that $\nu([-1,1]^d) =+\infty.$ 
We are going to show that $\tau=0$ almost surely, which implies that $\tau$ is independent from the first jump size. 
For any $a>0,$ define $\tau^a:=\inf\{t>0, \norm{\Delta X_t}\ge a\}.$ 
Then $\tau^a$ is also the first jump time of a (simple) Poisson process with decomposition $\int_0^\cdot \int_{\norm{x}\ge a}  \mu(\ud s,\ud x).$ Hence, for any $\e>0,$
\begin{align*}
    \P(\tau^a>\e)= \P\Big(\int_0^\e \int_{\norm{x}\ge a}  \mu(\ud s,\ud x)=0\Big)=\ue^{-\e \int_{\norm{x}\ge a}\nu(\ud x)},
\end{align*}
therefore $\P(\tau>\e)\leq  \P(\tau^a>\e)= \ue^{-\e \int_{\norm{x}\ge a}\nu(\ud x)}. $  
Letting $a\to0,$ we have for any $\e>0, $ that $\P(\tau>\e)=0,$ which implies $\tau=0$ almost surely. 
\end{proof}

\begin{theorem}\label{thm:optimility}
Let $X$ and $X^n$ be as in \eqref{levylim} and \eqref{levyapp}--\eqref{can-app-lev-meas}. 
Then we have the following inequality: 
\begin{align}\label{ineq optimility}
\cW_\rho\big( X^n,X\big)
    \ge c_T \Big(\int_{\norm{x}\leq \frac{1}{2n}} \norm{x}^2 \nu(\ud x)\Big)^{\frac{1}{2}},
\end{align}
where $c_T>0$. 
\end{theorem}

The proofs of this theorem and the next corollary are deferred to \cref{sec:main-proof}.
The next result is the main outcome of this section, and shows that the convergence rate obtained in Corollary \ref{cor-exp-rate} for the Wasserstein distance is indeed optimal.

\begin{corollary}
\label{cor:optimality}
Let $X$ and $X^n$ be as in \eqref{levylim} and \eqref{levyapp}--\eqref{can-app-lev-meas}, and assume that the Blumenthal--Getoor index satisfies $\beta_*<2$. 
Then, for any $0<\beta<\beta_*,$  holds
\begin{align}
   \limsup_{n\to\infty} n^{1-\frac{\beta}{2}} \cW_\rho\big( X^n,X\big) = +\infty.
\end{align} 
\end{corollary}

\subsection{The case $\beta=\beta_*$}

In this subsection, we will consider the case $\beta=\beta^*$ and show, using two examples, that an optimal rate cannot be determined in this case. 
 
\begin{Example}
Assume that the L\'evy measure equals $\nu(\ud x)= \sum^\infty_{i=1} \delta_{\frac{1}{i}}(\ud x)$, where $\delta$ is the Dirac measure, \textit{i.e.} $\delta_x(A)=1$ if and only if $x\in A.$ 
We consider the following series:
\begin{align}\label{harseries}
    \sum^\infty_{i=1} \frac{1}{i^p}.
\end{align}
This series converges if and only if $p>1.$ 
Hence, the Blumenthal--Getoor index is exactly $1.$ 
We also have that
\begin{align*}
\int_{\norm{x}\leq \frac{1}{n}} \norm{x}^2 \nu(\ud x) 
    = \sum^\infty_{i=n} \frac{1}{i^2}
    \leq  \sum^\infty_{i=n} \frac{1}{i(i-1)}
    = \frac{1}{n-1}.
\end{align*}
In the same way we get that $\int_{\norm{x}\leq \frac{1}{n}} \norm{x}^2 \nu(\ud x)\ge \frac{1}{n}.$ 
Hence, from Theorem \ref{thm:optimility} and the proof of Proposition \ref{appprocess}, we finally get that
\begin{align*}
    \frac{c_T}{\sqrt{n}} \leq \cW_\rho\big( X^n,X\big) \leq \frac{C_T}{\sqrt{n}}.
\end{align*}
which means that the optimal rate is exactly $n^{-\frac{1}{2}}.$ 
\end{Example}

\begin{Example}
Assume that the L\'evy measure equals $\nu(\ud x)= \sum^\infty_{i=1} \delta_{\frac{\sqrt{\ln(i)}}{i}}(\ud x)$. 
The series $\sum^\infty_{i=1} \big(\frac{\sqrt{\ln(i)}}{i}\big)^p$ converges if and only if $p>1.$ 
Hence, the Blumenthal--Getoor index is again equal to $1.$ 
By the same analysis as in the previous example, we get 
\begin{align*}
\int_{\norm{x}\leq \frac{1}{n}} \norm{x}^2 \nu(\ud x) 
    &=  \sum^\infty_{i=n} \frac{\ln(i)}{i^2}\ge  \sum^\infty_{i=n} \frac{\ln(i)}{i(i+1)}
    = \sum^\infty_{i=n} \ln(i)\Big(\frac{1}{i}-\frac{1}{i+1}\Big)\\
     &=\frac{\ln(n)}{n}+\sum^{\infty}_{i=n}\frac{\ln(i+1)-\ln(i)}{i+1}\ge \frac{\ln(n)}{n}.
\end{align*}
On the other hand, for $n\ge 2$ we have
\begin{align*}
\int_{\norm{x}\leq \frac{1}{n}} \norm{x}^2 \nu(\ud x) 
    &= \sum^\infty_{i=n} \frac{\ln(i)}{i^2}
    \leq  \sum^\infty_{i=n} \frac{\ln(i)}{i(i-1)}
    = \sum^\infty_{i=n}\ln(i)\Big( \frac{1}{i-1}-\frac{1}{i}\Big)\\
    &= \frac{\ln(n)}{n-1}+\sum^{\infty}_{i=n}\frac{\ln(i+1)-\ln(i)}{i}
    \overset{(c)}{\leq} \frac{\ln(n)}{n}+\sum^{\infty}_{i=n}\frac{1}{i^2}\\
    &\leq \frac{\ln(n)}{n}+\frac{1}{n-1}
    \leq \frac{2\ln(n)}{n},
 \end{align*}
where for inequality $(c)$ we have used the fact that $\ln(1+\frac{1}{i})\leq \frac{1}{i}.$
Therefore,
\begin{align*}
    \frac{c_T\sqrt{\ln(n)}}{\sqrt{n}} \leq\cW_\rho\Big( X^n,X\Big)\leq \frac{C_T\sqrt{\ln(n)}}{\sqrt{n}}.
\end{align*}
which means that the optimal rate is $\sqrt{\ln(n)}n^{-\frac{1}{2}}$ but not $n^{-\frac{1}{2}}.$
\end{Example}

\section{Proofs}
\label{sec:main-proof}

This section contains the proofs of the main results from Sections \ref{sec:main}, \ref{sec:approx-BSDE} and \ref{sec:optimality}.

\begin{proof}[Proof of Theorem \ref{maintheorem1}]
Let us consider the setting of Subsection \ref{set-2}, \textit{i.e.} we assume that the terminal random variables $\xi,\xi^n$ are arbitrary $\mathcal G_T$-,$\mathcal G_T^n$-measurable random variables, while the \lev process $X$ is approximated by $X^n$ in \eqref{levyapp} and its \lev measure satisfies $\nu^n(\ud x)=\ind{\norm{x}\ge \frac1n} \nu(\ud x)$.
Set $ \overline{U}^n_s(x)= U^n_s(x)\ind{\norm{x}\ge \frac1n}$. 
Since $\nu^n$ is supported on $\{\norm{x}\ge\frac1n\}$, we can directly verify that $\int_{\R^d} U^n_s(x) \widetilde \mu^n(\ud s,\ud x)=\int_{\R^d} \overline{U}^n_s(x) \widetilde \mu^n(\ud s,\ud x)$ and, using the Lipschitz property \ref{tF2}, we can show that $f(s,Y^n_s,U^n_s(\cdot))=f(s,Y^n_s,\overline{U}^n_s(\cdot)).$ 
This implies that $(Y^n,\overline{U}^n_s)$ is indistinguishable from the unique solution of (\ref{mainappeq}). 
Thus, we can write $ Y^n_t- Y_t$ as follows:
\begin{align}
 Y^n_t- Y_t
    &= \xi^n - \xi + \int_t^T f\big(s,Y^n_s,\overline{U}^n_s(\cdot)\big) - f\big(s,Y_s,U_s(\cdot)\big)\ud s \nonumber\\
    &\quad\quad - \int_t^T \int_{\R^d} \overline{U}^n_s(x) \widetilde \mu^n(\ud s,\ud x) + \int_t^T \int_{\R^d} U_s(x) \widetilde \mu(\ud s,\ud x) \nonumber\\
    &= \xi^n-\xi+\int_t^T f\big(s,Y^n_s,\overline{U}^n_s(\cdot)\big)- f\big(s,Y_s,U_s(\cdot)\big)\ud s
 -\int_t^T \int_{\R^d} \overline{U}^n_s(x)- U_s(x) \widetilde \mu(\ud s,\ud x).
 \label{ident_diff_Y}
\end{align}

\noindent \textit{Step 1:} In this step, we are going to compute an auxiliary bound.
Using It\^o's formula, we have for $b,\lambda>0$ that
\begin{align}
\E\big[ & \ue^{b s} \abs{ Y^n_s- Y_s}^2\big] 
+ \E \bigg[b \int_s^T \ue^{b t} \abs{ Y^n_{t-}- Y_{t-}}^2\ud t 
+ \int_s^T\int_{\R^d} \ue^{b t} \abs{\overline{U}^n_t(x)- U_t(x)}^2  \nu(\ud x)\ud t \bigg] 
\nonumber\\
&=\E\big[\ue^{b T} \abs{ \xi^n- \xi}^2\big] 
+ 2\E \bigg[\int_s^T \ue^{b t}(Y^n_{t-}- Y_{t-})\Big(f\big(t,Y^n_t,\overline{U}^n_t(\cdot)\big)- f\big(t,Y_t,U_t(\cdot)\big)\Big)\ud t \bigg]  
    \label{equ YnY_1}\\
& \leq \E\big[\ue^{b T} \abs{ \xi^n- \xi}^2\big] 
+ (2 + \lambda)L_f  \E \bigg[ \int_s^T \ue^{bt} |Y^n_t - Y_t|^2 \ud t \bigg]
\nonumber \\
&\quad  
+ \frac{L_f}{\lambda}  
\E \bigg[ \int_s^T \int_{\R^d}\ue^{bt} \abs{\overline{U}^n_t(x)- U_t(x)}^2 \nu(\ud x)\ud t \bigg],
\label{equ YnY_2}
\end{align}
where in \eqref{equ YnY_2} we used the Lipschitz property of the generator $f$ and afterwards Young's inequality for some $\lambda>0$. 
In addition, we used the fact that the Lebesgue measure is atomless, which allows to substitute within the Lebesgue--Stieltjes integrals the variable $Y^n_{t-}$, resp. $Y_{t-}$, with the variable $Y^n_{t}$, resp. $Y_{t}$, for every $t\in \mathbb{T}$.
Regarding the identity \eqref{equ YnY_1}, we used the fact that the stochastic integral is a true martingale.
Indeed, using the Burkholder--Davis--Gundy (BDG) inequality for $p=1,$ we have
\begin{align*}
&\E\Big[\sup_{0\leq s\leq T}\abs{\int_s^T \int_{\R^d}\ue^{b t} (Y^n_{t-}- Y_{t-}) \big(\overline{U}^n_t(x)- U_t(x)\big) \widetilde \mu(\ud t,\ud x)}\Big] \\
    &\hspace{1em}\leq 
    C\E\Big[\abs{\int_0^T \int_{\R^d}\ue^{2b t} \abs{Y^n_{t-}- Y_{t-}}^2\abs{ \overline{U}^n_t(x)- U_t(x)}^2  \mu(\ud t,\ud x)}^\frac{1}{2}\Big]\\
    &\hspace{1em}\leq C\ue^{bT} \E\Big[ 
    \big(\sup_{t\in\T}\abs{Y^n_{t}- Y_{t}}^2\big)^{\frac{1}{2}}
    \abs{\int_0^T \int_{\R^d} \abs{\overline{U}^n_t(x)- U_t(x)}^2  \mu(\ud t,\ud x)}^\frac{1}{2}\Big]\\
    &\hspace{1em}\leq 
    \frac{C\ue^{bT}}{2\gamma}\E\big[ \sup_{t\in\T}\abs{Y^n_{t}- Y_{t}}^2\big]
    +\frac{C\gamma\ue^{bT}}{2}
    \E\Big[ \int_0^T \int_{\R^d} \abs{\overline{U}^n_t(x)- U_t(x)}^2  \mu(\ud t,\ud x)\Big]\\
    &\hspace{1em}=\frac{C\ue^{bT}}{2\gamma}
    \E[ \sup_{t\in\T}\abs{Y^n_{t}- Y_{t}}^2]
    +\frac{C\gamma\ue^{bT}}{2}
    \E\Big[ \int_0^T \int_{\R^d}\abs{\overline{U}^n_t(x)- U_t(x)}^2  \nu(\ud x)\ud t\Big]
    <+\infty,
\end{align*}
where $C$ is the constant from the BDG inequality and $\gamma>0$. 

Now, using \eqref{equ YnY_2} for $\lambda = 2L_f$ and $b= (2+2L_f)L_f+\frac{1}{2}$ we have for every $s\in\mathbb{T}$
\begin{align}
    &\E\big[\ue^{b s} \abs{ Y^n_s- Y_s}^2\big] 
    + \frac{1}{2}\E \bigg[\int_s^T \ue^{b t} \abs{ Y^n_{t-}- Y_{t-}}^2\ud t\bigg] 
    + \frac{1}{2}\E \bigg[\int_s^T\int_{\R^d} \ue^{b t} \abs{\overline{U}^n_t(x)- U_t(x)}^2  \nu(\ud x)\ud t \bigg] \nonumber\\
    &\leq 
    \ue^{b T}  \E\big[\abs{ \xi^n- \xi}^2\big] \label{aux_norm_bound}.
\end{align}
\noindent \textit{Step 2:} Let us now derive an upper bound for the desired norms.
To this end, we have from \eqref{ident_diff_Y} 
\begin{align*}
    Y^n_t - Y_t 
    = \mathbb{E}\bigg[ 
    \xi^n-\xi
    +\int_t^T f\big(s,Y^n_s,\overline{U}^n_s(\cdot)\big)- f\big(s,Y_s,U_s(\cdot)\big)\ud s  
    \bigg| \mathcal{G}_t    \bigg],
\end{align*}
which further implies using Doob's inequality, Jensen's inequality and the Lipschitz property of the generator
\begin{align}
    \mathbb{E}\big[\sup_{t\in\mathbb{T}} |Y^n_t - Y_t|^2 \big]
    & \leq 
    \mathbb{E} \bigg[\sup_{t\in\mathbb{T}} \bigg(\mathbb{E}\bigg[ 
    |\xi^n-\xi|
    + \int_0^T |f\big(s,Y^n_s,\overline{U}^n_s(\cdot)\big)- f\big(s,Y_s,U_s(\cdot)\big)|\ud s  \bigg| \mathcal{G}_t \bigg]\bigg)^2\bigg] \nonumber\\
    &\leq 2\mathbb{E} \bigg[ 
    |\xi^n-\xi|^2 
    + 2L_f T
    \Big( \int_0^T |Y^n_s - Y_s|^2 \ud s + \int_0^T\int_{\mathbb{R}^d} |\overline{U}^n_s(\cdot) - U_s(\cdot)\big)|^2\nu(\ud x) \ud s\Big)
    \bigg] \nonumber \\
    &\leq
    C_{L_f,T} \mathbb{E}[ | \xi^n - \xi|^2],
\end{align}
where the last inequality is an outcome of \eqref{aux_norm_bound}, for some constant $C_{L_f,T}$ which depends on the Lipschitz constant $L_f$ and the time horizon $T$.
\end{proof}

In order to prove the main results of Section \ref{sec:approx-BSDE}, we need the following preparatory result.

\begin{lemma}\label{prior-estimate-2}
Consider the setting of Theorem \ref{maintheorem11}.
Then, there exists a constant $ C_{L_f,L_g,T},$ such that
\begin{align*}
     \sup_{t\in\T}\E\big[\abs{ Y^n_t}^2\big]^\frac{1}{2} \leq C_{L_f,L_g,T},
\end{align*}
and
\begin{align*}
      \int_0^T \Big( \abs{Y^n_t}^2 + \int_{\R^d} \E\big[\abs{ \overline{U}^n_t(x)}^2\big]  \nu(\ud x) \Big) \ud t \leq  C_{L_f,L_g,T}, \nonumber
\end{align*}
where again $\overline{U}^n_t(x)=U^n_t(x)\ind{\norm{x}\ge \frac1n}.$
\end{lemma}

\begin{proof}
We first rewrite \eqref{mainappeq1} as follows:
\begin{align}
 Y^n_t
    &= \xi^n  + \int_t^T f^n\big(s,Y^n_s,\overline{U}^n_s(\cdot)\big) - f^n\big(s,0,0\big)\ud s+\int_t^T  f^n\big(s,0,0\big)\ud s \nonumber\\
    &\quad\quad - \int_t^T \int_{\R^d} \overline{U}^n_s(x) \widetilde \mu^n(\ud s,\ud x). 
 \label{ident_diff_Y2}
\end{align}
Similarly to Step 1 in the proof of Theorem \ref{maintheorem1}, where now we consider only $Y_t^n$, 
we get that, 
\begin{align}
\E\big[ & \ue^{b s} \abs{ Y^n_s}^2\big] 
+ \E \bigg[b \int_s^T \ue^{b t} \abs{ Y^n_{t-}}^2\ud t 
+ \int_s^T\int_{\R^d} \ue^{b t} \abs{U^n_t(x)}^2  \nu^n(\ud x)\ud t \bigg] 
\nonumber\\
&=\E\big[\ue^{b T} \abs{ \xi^n}^2\big] 
+ 2\E \bigg[\int_s^T \ue^{b t}(Y^n_{t-}-0)\Big(f^n\big(t,Y^n_t,U^n_t(\cdot)\big)- f^n\big(t,0,0\big)\Big)\ud t \bigg] \nonumber\\
&\quad + 2\E \bigg[\int_s^T \ue^{b t}Y^n_{t-} f^n\big(t,0,0\big)\ud t \bigg] 
    \label{equ YnY_12}\\
& \leq \E\big[\ue^{b T} \abs{ \xi^n}^2\big] 
+ (3 + \lambda)L_f  \E \bigg[ \int_s^T \ue^{bt} |Y^n_t |^2 \ud t \bigg]+\ue^{bT}\int_s^T (f^n)^2(t,0,0)\ud t
\nonumber \\
&\quad  
+ \frac{L_f}{\lambda}  
\E \bigg[ \int_s^T \int_{\R^d}\ue^{bt} \abs{U^n_t(x)}^2 \nu^n(\ud x)\ud t \bigg],
\label{equ YnY_22}
\end{align}
and, for $b\ge (3+2 L_f)L_f+1,$ we arrive at
\begin{align}\label{prior-eq-2}
    \int_s^T \ue^{b t} \E[\abs{ Y^n_t}^2]\ud t+\ue^{b s} \E[\abs{ Y^n_s}^2]+\frac{1}{2}\int_s^T\ue^{b t} \int_{\R^d} \E\Big[\abs{ U^n_t(x)}^2\Big]  \nu^n(\ud x)\ud t\leq \ue^{b T} \E[\abs{ \xi^n}^2] + \ue^{bT} {C_T},
\end{align}
where $C_T:= 2\int_0^T f^2(t,0,0)\ud t+2\sup_{n\ge 0}c^2_n\ge \int_0^T (f^n)^2(t,0,0)\ud t $ due to \ref{tF4}.

The definition of $\overline{U}^n_t$ implies 
\[
    \int_{\R^d} \E\big[\abs{U^n_t(x)}^2\big]  \nu^n(\ud x)=\int_{\R^d} \E\big[\abs{\overline{U}^n_t(x)}^2\big]  \nu(\ud x).
\]
Since $s\in[0,T]$ is arbitrary and 
\[
    \sup_{n\ge 0}\E[\abs{\xi^n}^2]=\sup_{n\ge 0}\E[g^2(X_T^n)]\leq 2g^2(0)+L_g\E[\norm{X_T^n}^2]\leq 2g^2(0)+L_g\int_{\R^d}\norm{x}^2\nu(\ud x)<+\infty,
\]
we get from inequality \eqref{prior-eq-2} that
\begin{align*}
\sup_{t\in\T}\E\big[\abs{ Y^n_t}^2\big]^\frac{1}{2}
    \leq \sup_{t\in\T} \ue^{bt} \E\big[\abs{ Y^n_t}^2\big]
    \leq \ue^{b T} \E\big[\abs{\xi^n}^2\big] + { \ue^{bT}C_T}\leq  C_{L_f,L_g,T},
\end{align*}
and
\begin{align*}
\int_0^T \Big( \abs{ Y^n_t}^2 + \int_{\R^d} \E\big[\abs{\overline{U}^n_t(x)}^2\big]  \nu(\ud x) \Big)\ud t
    &\leq \int_0^T \Big( \ue^{b t} \abs{Y^n_t}^2 + \frac{1}{2}\ue^{b t} \int_{\R^d} \E\big[\abs{ U^n_t(x)}^2\big] \nu^n(\ud x) \Big) \ud t\\
    &\leq \ue^{b T} \E\big[\abs{\xi^n}^2\big] + \ue^{bT}{C_T}
     \leq  C_{L_f,L_g,T}. \qedhere 
\end{align*}
\end{proof}

\begin{proof}[Proof of Theorem \ref{maintheorem11}]

Analogously to the proof of Theorem \ref{maintheorem1}, using It\^o's formula and integrating on both sides over $(s,T]$, we get
\begin{multline}\label{equ YnY11}
\ue^{b T} \abs{ \xi^n- \xi}^2-\ue^{b s} \abs{ Y^n_s- Y_s}^2 = \\
    = b \int_s^T \ue^{b t} \abs{ Y^n_t- Y_t}^2\ud t - 2\int_s^T \ue^{b t}(Y^n_t- Y_t)\Big(f^n\big(t,Y^n_t,\overline{U}^n_t(\cdot)\big)- f\Big(t,Y_t,U_t(\cdot)\big)\Big)\ud t  \\
    + 2\int_s^T \ue^{b t} \int_{\R^d}(Y^n_{t-}- Y_{t-}) \big(\overline{U}^n_t(x)- U_t(x)\big) \widetilde \mu(\ud t,\ud x)+ \int_s^T\ue^{b t} \int_{\R^d} \abs{\overline{U}^n_t(x)- U_t(x)}^2  \nu(\ud x,\ud t).
\end{multline}    
Using \citet[Chapter 19]{cohen2015stochastic} or \citet[Chapter 4.1]{delong2013backward} together with Assumption \ref{G1}, which means we are in the Markovian setting, $Y^n_t$ can be expressed as $u^n(t, X^n_t)$ for some Lipschitz function $u^n$. 
Together with Assumption \ref{tF4}, this implies that $\overline{U}^n_t(x)=u^n(t,X^n_{t-}+x)-u^n(t,X^n_{t-})$ is uniformly Lipschitz continuous; see also \citet[Assumption 1, (iii)]{madan2015convergence}. 

Hence, we could  bound the term $(Y^n_t- Y_t)\Big(f^n\big(t,Y^n_t,\overline{U}^n_t(\cdot)\big)- f\big(t,Y_t,U_t(\cdot)\big)\Big)$ by
\begin{multline}\label{ineq Ynf111}
2(Y^n_t- Y_t)\Big(f^n\big(t,Y^n_t,\overline{U}^n_t(\cdot)\big)- f\big(t,Y_t,U_t(\cdot)\big)\Big)=\\
=2(Y^n_t- Y_t)\Big((f^n-f)\big(t,Y^n_t,\overline{U}^n_t(\cdot)\big) +f\big(t,Y^n_t,\overline{U}^n_t(\cdot)\big)-f\big(t,Y_t,U_t(\cdot)\big)\Big)\\
   \leq 2\abs{Y^n_t- Y_t} C_{n,t}
        + 2L_f\abs{Y^n_t- Y_t}^2+2L_f\abs{Y^n_t- Y_t}\Big(\int_{\R^d}\abs{\overline{U}^n_t(x)-U_t(x)}^2\nu(\ud x)\Big)^{\frac{1}{2}}\\
   \leq  \abs{Y^n_t- Y_t}^2+ C^2_{n,t}+L_f\Big[(2+\lambda)\abs{Y^n_t- Y_t}^2+\lambda^{-1}\int_{\R^d}\abs{\overline{U}^n_t(x)-U_t(x)}^2\nu(\ud x)\Big],
\end{multline}
where $ C_{n,t}=c_n\Big[1+\abs{Y^n_t}^2+\int_{\R^d} \abs{\overline{U}^n_t(x)}^2\nu (\ud x)\Big]^\frac{1}{2}$.
Then, taking an expectation in \eqref{equ YnY11} and using \eqref{ineq Ynf111}, we get that
\begin{align*}
\ue^{b T} \E[\abs{ \xi^n- \xi}^2]-\ue^{b s} \E[\abs{ Y^n_s- Y_s}^2]
    &\ge [b- (2+\lambda)L_f-1] \int_s^T \ue^{b t} \abs{ Y^n_t- Y_t}^2\ud t\\
     +(1-L_f\lambda^{-1})&\int_s^T\ue^{b t} \int_{\R^d} \E\Big[\abs{\overline{U}^n_t(x)- U_t(x)}^2\Big]  \nu(\ud x)\ud t- \int_s^T\E[ C^2_{n,t}]\ud t.
\end{align*}
As in \eqref{equ YnY_2}, setting $\lambda=2 L_f$, $b= (2+2 L_f)L_f+\frac{3}{2},$ and using Lemma \ref{prior-estimate-2}, we arrive at
\begin{align}\label{withoutsupYn1}
\ue^{b s} \E[\abs{ Y^n_s- Y_s}^2] + \frac{1}{2}\int_s^T\ue^{b t} \int_{\R^d} \E\Big[\abs{\overline{U}^n_t(x)- U_t(x)}^2\Big]  \nu(\ud x)\ud t
    \leq \ue^{b T} \E[\abs{ \xi^n- \xi}^2]+ C_{L_f,L_g,T}c_n^2.
\end{align}
Moreover, from equality \eqref{equ YnY11} and inequality \eqref{ineq Ynf111}, we have
\begin{multline*}
\ue^{b T} \abs{ \xi^n- \xi}^2-\ue^{b s} \abs{ Y^n_s- Y_s}^2 \ge \\
    \ge 2\int_s^T \ue^{b t} \int_{\R^d}(Y^n_{t-}- Y_{t-}) \big(\overline{U}^n_t(x)- U_t(x)\big) \widetilde \mu(\ud t,\ud x)\\
        - \frac{1}{2}\int_s^T\ue^{b t} \int_{\R^d} \abs{\overline{U}^n_t(x)- U_t(x)}^2  \nu(\ud x)\ud t-\int_s^T C^2_{n,t}\ud t.
\end{multline*}
Similarly to the proof of Theorem \ref{maintheorem1} once again, using the Burkholder--Davis--Gundy inequality for $p=1,$ we can also deduce that 
\begin{align*}
\E\big[ \sup_{t\in\T}\ue^{b t} \abs{ Y^n_t- Y_t}^2 \big] 
    &\leq 2 \ue^{b T} \E[\abs{ \xi^n- \xi}^2] \\
    &\quad +4C^2\E\Big[ \int_0^T \int_{\R^d} \ue^{b t}\abs{ \overline{U}^n_t(x)- U_t(x)}^2  \nu(\ud x)\ud t\Big]+2\int_0^T\E[ C^2_{n,t}]\ud t.
\end{align*}
Recalling \eqref{withoutsupYn1} and that Lemma \ref{appprocess} together with Assumption \ref{G1} yield 
\[
    \E[\abs{ \xi^n- \xi}^2] \leq L^2_g \ \E[\norm{X^n_T-X_T}^2] \leq \frac{(L_g C_\beta)^2 T}{n^{2-\beta}},
\]
allows us to conclude the proof.
\end{proof}

\begin{proof}[Proof of \cref{thm:optimility}]
For any $\Gamma^n\in\mathcal{H}(X^n|_{\mathbb{P}},X|_{\mathbb{P}})$, it is possible to find another probability space $(\hat{\Omega},\hat{\mathcal{F}},\hat{\mathbb{P}})$ such that $(\hat X^n|_{\hat{\mathbb{P}}}, \hat X|_{\hat{\mathbb{P}}}) \sim \Gamma^n $, where $\sim$ denotes equality in distribution, and we just need to verify that 
$$
\hat\E\Big[\sup_{t\in\T}\norm{\hat X_t^n-\hat X_t}^2\Big]^{\frac{1}{2}}\ge c_T \Big(\int_{\norm{x}\leq \frac{1}{2n}} \norm{x}^2 \nu(\ud x)\Big)^{\frac{1}{2}}.
$$  
Here $\hat\E$ denotes the expectation under $\hat{\mathbb{P}}$ and $c_T>0$ does not depend on the choice of $\Gamma^n.$
Let us denote the first jump size of $(\hat X_t)_{t\in\T}$ and $(\hat X^n_t)_{t\in\T}$ by $\Delta \hat X$ and $\Delta \hat X^n$ respectively.
Obviously $\Delta \hat X|_{\hat{\mathbb{Pnocomments}}}\sim \Delta  X|_{\mathbb{P}}$ and $\Delta \hat X^n|_{\hat{\mathbb{P}}}\sim \Delta  X^n|_{\mathbb{P}}.$ 
We also denote by $\hat \tau$ and $\hat \tau^n$ the first jump time of $\hat X$ and $\hat X^n.$
Then, we have
\begin{multline*}
\hat\E\Big[\sup_{t\in\T}\norm{\hat X_t^n-\hat X_t}^2\Big]
    \ge \hat\E\Big[\sup_{t\in\T}\norm{\hat X_t^n-\hat X_t}^2\Bi_{\{\hat \tau\leq T\}}\Big] \\
     \ge \hat\E\Big[\Big(\norm{\hat X_{\hat \tau^n}^n-\hat X_{\hat \tau^n}}^2\Bi_{\{\hat \tau>\hat \tau^n\}} + \norm{\hat X_{\hat \tau^n}^n-\hat X_{\hat \tau^n}}^2\Bi_{\{\hat \tau=\hat \tau^n\}} + \norm{\hat X_{\hat \tau}^n-\hat X_{\hat \tau}}^2\Bi_{\{\hat \tau<\hat \tau^n\}}\Big)\Bi_{\{\hat \tau\leq T\}}\Big] \\
     = \hat\E\Big[\Big(
         \norm{\Delta \hat X^n}^2\Bi_{\{\hat \tau^n< \hat \tau\}} 
         + \norm{\Delta \hat X^n-\Delta \hat X}^2\Bi_{\{\hat \tau^n= \hat \tau\}} 
         + \norm{\Delta \hat X}^2\Bi_{\{\hat \tau^n> \hat \tau\}}\Big)
            \Bi_{\{\hat \tau\leq T\}}\Big] \\
     \ge \hat\E\Big[\Big(
        \norm{\Delta \hat X^n}^2\Bi_{\{\hat \tau^n< \hat \tau\}} 
        + \norm{\Delta \hat X^n-\Delta \hat X}^2\Bi_{\{\hat \tau^n= \hat \tau\}} 
        + \norm{\Delta \hat X}^2\Bi_{\{\hat \tau^n> \hat \tau\}}\Big)
            \Bi_{\{\hat \tau\leq T,\norm{\Delta \hat X}\leq \frac{1}{2n}\}}\Big] \\
     \ge \hat\E\Big[\Big(\norm{\Delta \hat X^n}^2\Bi_{\{\hat \tau^n< \hat \tau\}} + \big(\norm{\Delta \hat X^n}^2-2\norm{\Delta \hat X}\norm{\Delta \hat X^n} + \norm{\Delta \hat X}^2\big)\Bi_{\{\hat \tau^n= \hat \tau\}}\\
     + \norm{\Delta \hat X}^2\Bi_{\{\hat \tau^n> \hat \tau\}}\Big)\Bi_{\{\hat \tau\leq T,\norm{\Delta \hat X}\leq \frac{1}{2n}\}}\Big] \\
     = \hat\E\Big[\Big(\norm{\Delta \hat X^n}^2\Bi_{\{\hat \tau^n< \hat \tau\}} + \norm{\Delta \hat X^n}\big(\norm{\Delta \hat X^n}-2\norm{\Delta \hat X}\big)\Bi_{\{\hat \tau^n= \hat \tau\}}\\
     + \norm{\Delta \hat X}^2\Bi_{\{\hat \tau^n\ge \hat \tau\}}\Big)\Bi_{\{\hat \tau\leq T,\norm{\Delta \hat X}\leq \frac{1}{2n}\}}\Big] \\
     \overset{(\alpha)}{\ge} \hat\E\Big[\Big(\norm{\Delta \hat X^n}^2\Bi_{\{\hat \tau^n< \hat \tau\}}
     + \norm{\Delta \hat X}^2\Bi_{\{\hat \tau^n\ge \hat \tau\}}\Big)\Bi_{\{\hat \tau\leq T,\norm{\Delta \hat X}\leq \frac{1}{2n}\}}\Big] \\
     \overset{(\beta)}{\ge} \hat\E\Big[\Big(\norm{\Delta \hat X}^2\Bi_{\{\hat \tau^n< \hat \tau\}} + \norm{\Delta \hat X}^2\Bi_{\{\hat \tau^n\ge \hat \tau\}}\Big)\Bi_{\{\hat \tau\leq T,\norm{\Delta \hat X}\leq \frac{1}{2n}\}}\Big] \\
     = \hat\E\Big[\norm{\Delta \hat X}^2\Bi_{\{\hat \tau\leq T,\norm{\Delta \hat X}\leq \frac{1}{2n}\}}\Big] \\ 
     \overset{(\gamma)}{=} \hat\P(\hat \tau\leq T)\hat\E\Big[\norm{\Delta \hat X}^2\Bi_{\{\norm{\Delta \hat X}\leq \frac{1}{2n}\}}\Big]
     =\hat\P(\hat \tau\leq T) \int_{\norm{x}\leq \frac{1}{2n}} \norm{x}^2 \nu(\ud x)\\
     =\P( \tau\leq T) \int_{\norm{x}\leq \frac{1}{2n}} \norm{x}^2 \nu(\ud x).
\end{multline*}
Inequalities $(\alpha)$ and $(\beta)$ hold since $\Delta \hat X^n$ is supported on $\norm{x}\ge \frac{1}{n}$ and under the event $\{\|\Delta \hat X\|\leq \frac{1}{2n}\}$, which together imply that $\norm{\Delta \hat X^n}-2\norm{\Delta \hat X}\ge 0$ and $\norm{\Delta \hat X^n}\ge \norm{\Delta \hat X}$. 
Equality $(\gamma)$ holds since $\hat X|_{\hat{\mathbb{P}}}\sim X|_{\mathbb{P}},$ which implies the independence between $\Delta \hat X$ and $\hat \tau$ from Lemma  \ref{lem indep}.
Notice that $\P( \tau\leq T)$ does not depend on the choice of $\Gamma^n$. 
Hence, setting $c_T= \sqrt{\P( \tau\leq T)}$ concludes the proof.
\end{proof}

\begin{proof}[Proof of \cref{cor:optimality}]
Using Theorem \ref{thm:optimility}, we just need to prove, for any $0<\beta<\beta_*,$ that
\begin{align*}
    \limsup_{n\to\infty} n^{2-\beta} \int_{\norm{x}\leq \frac{1}{2n}} \norm{x}^2 \nu(\ud x)=+\infty.
\end{align*}
From the definition of the Blumenthal--Getoor index we have, for any $0<\beta<\beta_*,$ that
 \begin{align*}
     \int_{\norm{x}\leq \frac{1}{2}} \norm{x}^{\beta} \nu(\ud x)=+\infty.
 \end{align*}
Let $A_i:=\int_{\norm{x}\leq \frac{1}{2i}} \norm{x}^2 \nu(\ud x),$ then we can decompose $\int_{\norm{x}\leq \frac{1}{2}} \norm{x}^{\beta} \nu(\ud x)$ as follows:
 \begin{align*}
     \int_{\frac{1}{2(n+1)}\leq\norm{x}\leq \frac{1}{2}} \norm{x}^{\beta} \nu(\ud x)
     &=\sum_{i=1}^n\int_{ \frac{1}{2(i+1)}<\norm{x}\leq \frac{1}{2i}} \norm{x}^{\beta} \nu(\ud x)
     =\sum_{i=1}^n\int_{ \frac{1}{2(i+1)}<\norm{x}\leq \frac{1}{2i}} \norm{x}^2 \norm{x}^{\beta-2}\nu(\ud x)\\
     &\leq \sum_{i=1}^n(2(i+1))^{2-\beta}\int_{ \frac{1}{2(i+1)}<\norm{x}
     \leq \frac{1}{2i}} \norm{x}^2 \nu(\ud x)
     = \sum_{i=1}^n(2(i+1))^{2-\beta}(A_i-A_{i+1})\\
     &\leq 4^{2-\beta} A_1+ \sum_{i=2}^n\Big[(2(i+1))^{2-\beta}-(2i)^{2-\beta}\Big]A_i-(2(n+1))^{2-\beta}A_{n+1}\\
     &\leq 4^{2-\beta} A_1+ \sum_{i=2}^n\Big[(2(i+1))^{2-\beta}-(2i)^{2-\beta}\Big]A_i\\
     &= 2^{2-\beta}\Big\{ 4 A_1 +\sum_{i=2}^n i^{2-\beta}\Big[\Big(1+\frac{1}{i}\Big)^{2-\beta}-1\Big]A_i\Big\}\\
    &\leq 2^{2-\beta}\Big\{ 4 A_1+ \sum_{i=2}^n i^{2-\beta}\Big[\Big(1+\frac{1}{i}\Big)^2-1\Big]A_i\Big\}\\
     &\leq 2^{2-\beta}\Big\{ 4 A_1+
     \sum_{i=2}^n i^{2-\beta}\Big[\frac{2}{i}+\frac{1}{i^2}\Big]A_i\Big\}\\
     &\leq 2^{2-\beta}\Big\{ 4 A_1+
     \sum_{i=2}^n i^{2-\beta}\frac{3}{i}A_i\Big\}\\
     &\leq 2^{2-\beta}\Big\{ 4 A_1+3\sum_{i=2}^n i^{1-\beta}A_i\Big\}\\
     &\leq 2^{2-\beta}3\sum_{i=1}^n i^{1-\beta}A_i.
\end{align*}
Hence, $\int_{\norm{x}\leq \frac{1}{2}} \norm{x}^{\beta} \nu(\ud x)=+\infty$ implies $\sum_{i=1}^{\infty}i^{1-\beta}A_i=+\infty$ for any $0<\beta<\beta_*.$ 

Next, we are going to show that 
\begin{align*}
    \limsup_{n\to\infty} n^{2-\beta} \int_{\norm{x}\leq \frac{1}{2n}} \norm{x}^2 \nu(\ud x)>0.
\end{align*}
Otherwise, there exists a $0<\beta_0<\beta_*,$ such that $ \limsup_{n\to\infty} n^{2-\beta_0} \int_{\norm{x}\leq \frac{1}{2n}} \norm{x}^2 \nu(\ud x)=0.$  
This implies, for any $i\ge 1,$ that
\begin{align*}
  A_i= \int_{\norm{x}\leq \frac{1}{2i}} \norm{x}^2 \nu(\ud x)\leq \frac{\e_i}{i^{2-\beta_0}}.
\end{align*}
with $\lim_{i\to\infty}\e_i=0.$
This means for any $0<\beta<\beta_*,$
\begin{align*}
   \sum_{i=1}^{\infty}i^{1-\beta}A_i\leq  \sum_{i=1}^{\infty} i^{-1-(\beta-\beta_0)}.
 \end{align*}
By choosing $\beta_0<\beta<\beta_*,$ this leads to $ \sum_{i=1}^{\infty}i^{1-\beta}A_i<+\infty,$ which is a contradiction. 
Therefore, for any $0<\beta<\beta_*,$ holds
\begin{align*}
    \limsup_{n\to\infty} n^{2-\beta} \int_{\norm{x}\leq \frac{1}{2n}} \norm{x}^2 \nu(\ud x)>0,
\end{align*}
hence we have that
\begin{equation*}
\limsup_{n\to\infty} n^{2-\beta} \int_{\norm{x}\leq \frac{1}{2n}} \norm{x}^2 \nu(\ud x)
    =\limsup_{n\to\infty} \underbrace{n^{\frac{-\beta+\beta_*}{2}}}_{+\infty}\underbrace{n^{2-\frac{\beta+\beta_*}{2}} \int_{\norm{x}\leq \frac{1}{2n}} \norm{x}^2 \nu(\ud x)}_{>0}
    =+\infty.\qedhere
\end{equation*}
\end{proof}

\appendix

\section{Approximation of \lev processes by random walks}
\label{appendix:neg-lev-rand-walk}

In this appendix, we show a negative result about the approximation of \lev processes by random walks, namely that it is not possible to approximate a pure-jump L\'evy process in the uniform norm using a random walk. 
Indeed, the following simple counterexample demonstrates that we cannot approximate even a (plain) Poisson process in the uniform norm using a random walk approximation. 
In case the topology is weakened to the Skorokhod $J_1-$topology, then it is possible to approximate the desired process by a random walk; see \emph{e.g.} \citet[Lemma 3]{lejay2014numerical} for the case of a Poisson process.

Let us denote by $\{k_n\}_{n\ge 1}$ an increasing positive integer sequence tending to $+\infty$ as $n\to +\infty.$  
Let $(Y_{ni})_{1\leq i\leq k_n}$ denote random variables (not necessarily i.i.d.),  and define the partial sum process as
\begin{align*}
    S_n(t):=\sum_{i=1}^{\upi{k_nt}}Y_{ni}.
\end{align*}
Let $N=(N_t)_{t\in \mathbb T}$ denote a Poisson process with rate $\lambda=1$. 
Then we have the following inequality.

\begin{prop}
The uniform distance between  $(N(t))_{t\in\T}$ and $(S_n(t))_{t\in\T}$ is positive, for any $T>0$, i.e.
\begin{align*}
    \E\Big[\sup_{t\in\T}\abs{N(t)- S_n(t)}\Big]\ge \frac{1-\ue^{-T}}{2}.
\end{align*}
\end{prop}

\begin{proof}
Denote $t_i(n):=i/k_n$ for $i=0,1,\dots$. 
Notice that the jump times of $S_n$ belong to the set $\mathcal{T}_n=\{t_0(n),t_1(n),\dots\}.$  
Let $\tau$ be the first jump time of the Poisson process $N.$ 
Since $t_0(n),\ t_1(n),\dots$ are all deterministic times, we have that $\mathbb{P}(\tau\in \mathcal{T}_n)=0.$ 
Therefore
\begin{align*}
\E\Big[\sup_{t\in\T}\abs{N(t)- S_n(t)}\Big]
    &\overset{\phantom{(A1)}}{\ge} \E\Big[\sup_{t\in\T}\abs{N(t)- S_n(t)}\ind{\tau\leq T,\tau\notin \mathcal{T}_n}\Big] \\
    &\overset{\phantom{(A1)}}{\ge} \E\Big[\max\Big\{\abs{N(\tau)- S_n(\tau)},\abs{N(\tau-)- S_n(\tau-)}\Big\}\ind{\tau\leq T,\tau\notin \mathcal{T}_n}\Big] \\
    &\overset{(A1)}{=} \E\Big[\max\Big\{\abs{N(\tau)- S_n(\tau)},\abs{N(\tau-)- S_n(\tau)}\Big\}\ind{\tau\leq T,\tau\notin \mathcal{T}_n}\Big] \\
    &\overset{(A2)}{=} \E\Big[\max\Big\{\abs{N(\tau-)+1- S_n(\tau)},\abs{N(\tau-)- S_n(\tau)}\Big\}\ind{\tau\leq T,\tau\notin \mathcal{T}_n}\Big] \\
    &\overset{\phantom{(A1)}}{\ge} \E\Big[\frac{1}{2}\Big\{\abs{N(\tau-)+1- S_n(\tau)}+\abs{N(\tau-)- S_n(\tau)}\Big\}\ind{\tau\leq T,\tau\notin \mathcal{T}_n}\Big] \\
    &\overset{\phantom{(A1)}}{\ge} \frac{1}{2}\E[\ind{\tau\leq T,\tau\notin \mathcal{T}_n}]
    = \frac{1}{2}\E[\ind{\tau\leq T}]
    = \frac{1-\ue^{-T}}{2}. 
\end{align*}
Equality $(A1)$ follows since $S_n(t)$ does not jump at $\tau$ for $\tau\notin \mathcal{T}_n,$ while $(A2)$ holds since $\tau$ is the first  jump time of $N(t)$, which implies that $N(\tau)=N(\tau-)+1.$ \end{proof}



\end{document}